\documentclass{amsart}
\newcommand{\T}{\mathcal{T}}
\renewcommand{\S}{\mathcal{S}}
\newcommand{\B}{\mathcal{B}}
\newcommand{\Z}{\mathbb{Z}}
\newcommand{\R}{\mathbb{R}}
\newcommand{\N}{\mathbb{N}}
\newcommand\mute[1]{}
\newcommand{\sdot}{\! \cdot \!}
\usepackage{braket,color}
\usepackage{amssymb,amscd,amsmath,amsthm,url,braket,microtype}
\usepackage[all]{xy}
\usepackage[pdftex]{graphicx}
\usepackage{dsfont}
\usepackage{tikz}

\usepackage[margin=1.3in]{geometry}

\DeclareMathOperator{\Av}{\rm GenCon}

\def \ssm{\smallsetminus}
\DeclareMathOperator{\Symm}{\rm Symm}
\newcommand{\Spos}{S_{\rm pos}}
\newcommand{\Sneg}{S_{\rm neg}}

\newtheorem {Theorem}                    {Theorem}
\newtheorem {Proposition}[Theorem]       {Proposition}

\newtheorem*{Question}   {Question}

\theoremstyle{definition}

\theoremstyle{remark}
\newtheorem{Remark}[Theorem]	{Remark}
\newtheorem{Example}[Theorem]	{Example}
\begin{document}

\title{Conjugation curvature for Cayley graphs}
\author{Assaf Bar-Natan, Moon Duchin, and Robert Kropholler}
\thanks{The authors gratefully acknowledge support from U.S. National Science Foundation grants 
DMS 1107452, 1107263, 1107367 ``RNMS: GEometric structures And Representation varieties" 
(the GEAR Network).  MD is supported by NSF DMS-1255442.
MD thanks Ruth Charney, Ilya Kapovich, and Jennifer Taback for very useful conversations.
We thank an anonymous referee for careful reading and for several suggestions that made the paper considerably stronger.}
\begin{abstract}
We introduce a notion of Ricci curvature for Cayley graphs that can be thought of as ``medium-scale"
because it is neither infinitesimal nor asymptotic, but based on a chosen finite radius parameter. 
We argue that it gives the foundation for a definition of Ricci curvature well adapted to geometric group theory, beginning by observing that the sign can easily be characterized in terms of conjugation in the group.
With this {\em conjugation curvature} $\kappa$, abelian groups are identically flat, and in the other direction we show that $\kappa\equiv 0$ implies the group is virtually abelian.  
Beyond that, $\kappa$
captures known curvature phenomena in right-angled Artin groups (including free groups)
and nilpotent groups, and has a strong relationship to other group-theoretic notions like growth rate and dead ends.  We study dependence on generators and behavior under embeddings, and close with 
directions for further development and study.  
\end{abstract}
\maketitle
\tableofcontents

\section{Introduction}

The emergence in the last several decades of a theory of large-scale negative curvature, or 
$\delta$-hyperbolicity, has provided an enormous breakthrough and a powerful tool for the 
theory of infinite groups and graphs.
The existence of any $\delta\ge 0$ for which a Cayley graph satisfies this family of metric conditions 
implies a host of consequences for the geometry of the graph and the spaces acted on by 
the group.  However, for large finite graphs, this theory is not applicable---every bounded-diameter
metric space is $\delta$-hyperbolic for $\delta \geq \mathrm{diam}$, so hyperbolicity provides no information.  
On the other hand, classical manifold notions of curvature are defined infinitesimally, so do not make 
sense for graphs and networks.  
This work was initially motivated by the desire to find a suitable curvature tool for the analysis of large but finite graphs or networks.

In the 2000s, Yann Ollivier defined a notion of metric Ricci curvature---at finite scales, not 
asymptotic or infinitesimal---for graphs
and other non-manifold geometries \cite{ollivier2007,ollivier2009,oll-surv,oll-curvsurv}.  
To do this, he offered a geometric interpretation of classical
Ricci curvature as follows:  curvature measures the extent to which corresponding points on 
spheres are {\em closer together} or {\em farther apart} than the centers of the spheres.  (Negative Ricci 
curvature occurs when corresponding points are farther and positive curvature when they are closer.)  
The ability to define {\em corresponding} points on spheres in the manifold
setting relies on parallel transport, so his generalized definition replaces the comparison 
distance with transportation distance, $L^1$ Wasserstein distance to be precise.\footnote{We note that 
Ollivier's was one of several proposed discretizations of Ricci curvature formulated in the same time period; 
other quite different definitions can be found in sources such as \cite{forman,yau}.}
With Ollivier's definition, he can show that in the manifold setting his metric Ricci curvature agrees with classical Ricci curvature on small scales, to first order.

\begin{figure}[ht]
\begin{tikzpicture}

\begin{scope}[scale=.85]
\filldraw (0,0) circle (.07);
\draw (0,0) circle (1)  (0,0) -- (79:1)   (0,0)--(109:1)  (0,0)--(5:1);
\filldraw [red] (79:1) circle (.07);   \filldraw [blue] (109:1) circle (.07);   \filldraw [green!50!black] (5:1) circle (.07);
\begin{scope}[xshift=5cm, yshift=2cm, rotate=-54]
\filldraw (0,0) circle (.07);
\draw (0,0) circle (1)  (0,0) -- (79:1)   (0,0)--(109:1)  (0,0)--(5:1);
\filldraw [red] (79:1) circle (.07);   \filldraw [blue] (109:1) circle (.07);   \filldraw [green!50!black] (5:1) circle (.07);
\end{scope}
\draw [dashed] (0,0).. controls (1,4) and (3,1)  ..(5,2);
\end{scope}
\begin{scope}[xshift=8cm]
\draw (-1,0,0)--(4,0,0)--(4,0,-1)--(4,2,-1)      (0,-1,0)--(0,1,0)  (0,0,-1)--(0,0,1)  (3,1,-1)--(5,1,-1)  (4,1,-2)--(4,1,0);
\filldraw (0,0,0) circle (.07)  (4,1,-1) circle (.07);  \filldraw [red] (0,0,-1) circle (.07)  (4,1,-2) circle (.07); 
\filldraw [blue] (1,0,0) circle (.07)  (5,1,-1) circle (.07);  \filldraw [green!50!black] (0,1,0) circle (.07)  (4,2,-1) circle (.07);
\end{scope}
\end{tikzpicture}
\caption{(Left:) In a manifold, points on spheres centered at the endpoints have a correspondence by parallel transport along a geodesic.
(Right:) In the Cayley graph for $\Z^3$, the colors mark corresponding neighbors of $e$ and $a^4bc$.}
\end{figure}

On the other hand, optimal transport is computationally expensive, making Ollivier's definition challenging 
to work with.  
The current paper is founded on the observation that Cayley graphs of finitely generated groups,
because of the edge labelings, provide a setting in which transportation distance is not needed, because
any two points' neighbors can be put in natural correspondence.  Thus we arrive at a new
definition of Ricci curvature in Cayley graphs, which parallels manifold Ricci curvature more closely 
than the Ollivier adaptation.  We call this {\em conjugation curvature} because it has a simple interpretation in terms of conjugation and word length. Running through a range of standard constructions from geometric group theory, we demonstrate the viability of conjugation curvature as the ``right"  Ricci curvature for discrete groups, offering suggestive connections between geometric and algebraic properties.  
This paper, and several successors written since its circulation as a preprint, have begun to elaborate the theory of conjugation curvature for discrete groups.

\smallskip

\paragraph{\bf Main results about conjugation curvature.}
\begin{itemize}
\item It is immediate that abelian groups have $\kappa\equiv 0$.  On the other hand, if any generating set has 
$\kappa=0$ at each generator, then $G$ is virtually abelian.  Furthermore, all virtually abelian groups 
have generating sets that make them ``virtually flat."  (Theorems~\ref{thm:vabelian}-\ref{thm:vabelian2})
\item Right-angled Artin groups with standard generators have zero curvature at central elements and negative curvature everywhere else.  In particular non-elementary free groups are negatively curved.
(Proposition~\ref{prop:RAAGs})
\item In the Heisenberg group, points of positive, zero, and negative curvature each have positive density. 
(Theorem~\ref{thm:heis})
\item If any Cayley graph $\Gamma(G,S)$ has positive curvature at all points outside of a ball, then $G$ is a finite group.  (Theorem~\ref{thm:finiteness})
\item If any Cayley graph $\Gamma(G,S)$ has negative curvature at all points outside of a ball, then $G$ has exponential growth.
(Nguyen--Wang \cite{expgrowth})
\item Dead-end elements give non-negative curvature and can often be leveraged to find positive curvature. (Kropholler--Mallery \cite{krophollermallery})
\end{itemize}

At the same time, it is important to understand the limitations of group definitions that are  
sensitive to the choice of generators.  
Theorems~\ref{thm:cat0}--\ref{thm:neg-embed2} sound notes of caution about interpreting 
isolated points of positive or negative curvature and about the behavior under embeddings.

\subsection{Definitions}

We will write $(G,S)$ for a group $G$ with generating set $S$, and $e\in G$ 
for the identity element.
We will assume throughout that $|S|<\infty$, $S=S^{-1}$,
and $e\notin S$.  We will use $\Gamma$ for graphs, in particular for the defining graph of a right-angled
Artin group $A_\Gamma$ (defined below in Example~\ref{ex:RAAG})
and for the Cayley graph  $\Gamma(G,S)$ associated to $(G,S)$.  
  For $x\in G$, let $|x|$ denote its length in the $(G,S)$ word metric.
  We write $B_r(x)=\{g\in G : |x^{-1}g|\le r\}$ and $S_r(x)=\{g\in G : |x^{-1}g|=r\}$ for the 
  ball and sphere of radius $r$ centered at $x$, respectively.
  
Let $\T_r(x,y)$ be the $L^1$ transportation distance for the uniform measure
on $B_r(x)$ and $B_r(y)$: it is the 
``earth mover's distance," or the infimal distance for any plan of moving mass from one probability distribution to the other.  
(See \cite{ollivier2009}.)
Using this, Ollivier defines what we will call the {\em transportation curvature} 
$\kappa_r^\T(x,y)=\frac{d(x,y)-\T_r(x,y)}{d(x,y)}$.

In a Cayley graph $\Gamma(G,S)$, note that there is a natural correspondence between neighbors of $x$ and neighbors 
of $y$ given by matching $xa$ with $ya$ for $a\in S$; more generally we can compare $xw$ with $yw$ for any 
word $w\in G$.
Thus let $\S_r(x,y)$ be the radius-$r$ spherical comparison distance in the Cayley graph $\Gamma(G,S)$:
$$\S_r(x,y):= \frac{1}{|S_r|} \sum_{w\in S_r} d(xw,yw),$$
or similarly  $\B_r(x,y)$ on balls rather than spheres.
 Thus  we can define a new Ricci curvature on groups via 
$\kappa^\B_r(x,y)=\frac{d(x,y)-\B_r(x,y)}{d(x,y)}$ or $\kappa^\S_r(x,y)=\frac{d(x,y)-\S_r(x,y)}{d(x,y)}$.

Note that $\kappa^\B_r(x,y)=\kappa^\B_r(e,x^{-1}y)$, so to study 
the curvature in a group it suffices to fix one of the points at the identity.  Thus we will write $\kappa^\B_r(g)$ 
to mean $\kappa^\B_r(e,g)$, and similarly for the spherical conjugation curvature $\S$.
For all three definitions, when the parameter $r$ is omitted, we are considering the case $r=1$.
In this paper we will focus on the spherical conjugation curvature with $r=1$, so we will 
simplify the notation even further, writing 
$$\kappa(g)=\kappa^\S_1(e,g).$$

One observes that $\S_1(g)$ is the 
the average word length of a conjugate of $x$ by a generator, so we will introduce the notation
$\Av(g)=\Av(G,S,g):=\frac{1}{|S|}\sum\limits_{a\in S} |a^{-1} g a| = \S_1(e,g).$
Our notion of Ricci curvature on groups is then given by setting $\kappa(e)=0$, and otherwise\footnote{We adopt the denominator $|g|$ to facilitate comparison to Ollivier's transportation curvature, but other normalizations are possible.  In any event, we mostly focus on the sign of $\kappa$.}
$$\kappa(g) =  \frac{|g|-\Av(g)}{|g|}.$$

Negative Ricci curvature (with respect to spheres of radius one) 
occurs at a group element $g$ when its conjugates by generators are,
on average, longer in word length than $g$ itself.  Respectively, positive curvature occurs when conjugation
shortens the length, and zero curvature when it is unchanged.
This motivates us to call $\kappa$ the  {\em conjugation curvature}.

\subsection{Properties and examples}

\begin{Proposition}[Basic properties of conjugation curvature] \label{properties} \

\begin{enumerate}
 \item \label{central} The curvature is zero at central elements, because
$g\in Z(G)\implies \Av(g)=|g|$.  
Indeed, all conjugation curvatures are identically zero on the center,
for balls or spheres of any radius.
This is because $d(xa,ya)=d(ax,ay)=d(x,y)$, so corresponding points have the same distance 
as the centers, giving $\S=\B=d$ and so $\kappa^\B_r=\kappa^\S_r\equiv 0$ for all $r$.
In particular, abelian groups with any generating sets have identically zero curvature.
\item \label{kappa0} For finite groups, we can take all elements as generators 
($S=G\ssm\{e\}$) to make $\kappa\equiv 0$, because 
 the Cayley graph is a complete graph so all distances are $1$.

\item For $r=1$, the value $\B$ is obtained as a weighted average of the value $\S$ with $d$,
the distance of the centers $x$ and $y$.  This means that  the sign of the conjugation curvature is the same whether
one considers balls or spheres of radius $1$.
\item Since $\T$ is an inf, and correspondence defines one of the competing transportation plans,
we have $\T\le \B$, so $\kappa_\T\ge \kappa_\B$ in general.
This means that for abelian groups, the transportation curvature is non-negative.
 \item\label{avgen} $|a^{-1} x a|$ is integer-valued and equals zero only for $x=e$, so  
 $\Av(e)=0$ and  $\Av(x)\ge 1$ for non-identity elements.
 Generators $a\in S$ satisfy  $\kappa(a)=1-\Av(a)$, so  $\kappa(a)\le 0$.

\item Suppose that $\{e,w,w^2,w^3,\hdots\}$ is geodesic in a group 
  $G$, meaning that $|w^k|=k|w|$ for all $k$.
  It follows that $\kappa(w)\le 0$. 
  Otherwise, there exists some generator $a\in S$ for which 
  $|a^{-1}wa|=d(a,wa)\le |w|-1$. In that case we would have
  $|w^k|\le |a^{-1}w^k a|+2\le k|w|-k+2$,
  which for $k\ge 3$ contradicts the assumption.

\item\label{2/d} For any $x,y$ and any $w\in B_r$, we have $d(x,y)-2r\le d(xw,yw) \le d(x,y)+2r$, 
and so the average is also in that range, which means that the correspondence curvature
is bounded:  $-2r/d \le \kappa^\B_r, \kappa^\S_r \le 2r/d$ for $d=d(x,y)$.  In particular if you fix $r$ and make $d$ large, 
the curvature becomes pinched near zero.
When $|w|\le 2$, we get $\kappa(w)\le 2/|w| \le 1$.  Putting these together,
we get $\kappa\le 1$ for every group element.
\end{enumerate}
\end{Proposition}

\begin{Proposition}[Damping out]
If $(G,S)$ is an infinite group, then 
the average curvature over the ball of radius $n$ tends to zero as $n\to\infty$.
\end{Proposition}

\begin{proof}
	The average curvature in the group $G$ is the limit $$\lim\limits_{n\to \infty}\frac{\sum_{g\in B_n}\kappa(g)}{|B_n|}.$$ 
Since $\kappa(e)=0$ and otherwise
$|\kappa(g)|\leq \frac{2}{|g|}$, we see that for any fixed $k$ we have 
	$$\lim\limits_{n\to \infty}\frac{\sum_{g\in B_n}\kappa(g)}{|B_n|} \leq \lim\limits_{n\to \infty}\frac{|B_k| + \frac{2}{k}|B_n\smallsetminus B_k|}{|B_n|} \leq \frac{2}{k}.$$ 
Since this holds for arbitrary $k$,  the average curvature tends to 0.
\end{proof}

\begin{Example}[Dependence on generators in the symmetric group]

We will illustrate a simple observation about the dependence on choice of generators
by considering the case of the symmetric group $\Symm(n)$.
Recall from Proposition~\ref{properties}\eqref{kappa0} that with respect to $S=G\ssm \{e\}$, we get $\kappa\equiv 0$.
If we focus, for example, on the group element $\sigma=(1\dots n)\in\Symm(n)$, then we will be able to
manipulate the sign of $\kappa(\sigma)$ by building an appropriate generating set.

\begin{Proposition}[Manipulating the sign]\label{symm-sign}\label{prop:symm-sign} 
Let $\sigma=(1\dots n)$ be the basic $n$-cycle in 
the symmetric group $\Symm(n)$.  
Given any  $\epsilon>0$, for sufficiently large  $n$ there exist two generating sets
$\Spos$ and $\Sneg$
 for $\Symm(n)$ such that 
 \begin{itemize}
\item with respect to $\Spos$, we have  $|\sigma|=2$,  $\Av(\sigma)<1+\epsilon$, and 
  $\kappa(\sigma)>\frac{1-\epsilon}{2}$ \ ; \quad  and 
\item with respect to $\Sneg$, we have  $|\sigma|=1$, $\Av(\sigma)>3-\epsilon$, and 
$\kappa(\sigma)<{-2+\epsilon}$.
\end{itemize}
\end{Proposition}
\begin{proof}
	Let $\Spos(n) = \Symm(n)\ssm\{e, \sigma\}^\pm$, so that $|\Spos|=n!-3$. 
	Consider the action of $\Symm(n)$ 
  	on itself by conjugation. We must consider the length $s^{-1}\sigma s$ for $s\in S$. If $s^{-1}\sigma s\in \{\sigma, \sigma^{-1}\}$, then we can see $|s^{-1}\sigma s| = 2$. This happens if $s$ commutes with $\sigma$ or conjugates it to $\sigma^{-1}$ in the first case there are $n-3$ elements of $S$ with this property. In the second case there are $n$ elements of $S$. In all other cases $s^{-1}\sigma s\in S$ and so has length 1. 
 	Thus we have
$$\Av(\sigma) = \frac{1}{n!-3}\bigl( (n!-3-n-(n-3)\bigr) + 2(n+n-3)\bigr)=\frac{n!+2n-6}{n!-3} \longrightarrow 1.$$

	Let $\Sneg(n) = \{(12), (23), \dots(n\! -\! 1 \  n), \sigma^\pm \}$, so that $|\Sneg|=n+1$. 
	Consider $|s^{-1}\sigma s|$ for $s\in \Sneg$. If $s =\sigma^{\pm}$, then $|s^{-1}\sigma s| = 1$. Since $\sigma$ is not fixed by any transposition we see that $|s^{-1}\sigma s| = 3$ otherwise. Thus, 
	$$\Av(\sigma) = \frac{2+3(n-1)}{n+1} \longrightarrow 3. \qedhere$$ 
\end{proof}

\end{Example}

\begin{Example}[Free groups with standard generators]\label{freecurv}
  Let $F_n$ be the free group with $n$ standard generators, and let 
  $w=a_1\cdots a_m$ be a freely reduced word of length $m$, with $a_i$ not necessarily distinct. If this spelling is not 
  cyclically reduced, 
  i.e.,  $a_1= a_m^{-1} =x$, then 
  $\Av(w) = \frac{(2n-1)(m+2)+(m-2)}{2n}$ 
  where the first term corresponds to conjugation by all generators 
  other than $x$, and the second term corresponds to conjugation by $x$. 
  If $w$ is cyclically reduced, then 
  $\Av(w) = \frac{(2n-2)(m+2)+m+m}{2n}$. In either case, 
  $\Av(w) = m+2-\frac{2}{n}$. Thus, 
  $\kappa(w) = 1- \frac{\Av(w)}{|w|}=\frac{2}{mn}-\frac{2}{m}
  =\frac{2-2n}{mn}$ 
  which is negative for $n\ge 2$, and $0$ for $n=1$ (the case of the elementary free group $\Z$).

We have established that $\Av(w)$ and $\kappa(w)$ both depend only on $|w|$ and the rank of the free group.
We get {\em negative curvature at all non-identity points} in free groups of rank at least two.
\end{Example}

\begin{Example}[$F_2\times \Z$]

 If $G = F_2\times \Z=\braket{a,b,z : [a,z]=[b,z]=e}$, 
      with $S = \{a,b,z\}^\pm$, then a 
      straightforward calculation shows that 
      $\Av(a) =\Av(b)= \frac{5}{3}$, so $\kappa(a) =\kappa(b)= -\frac{2}{3}$, as follows.
Conjugating $a$ by $a^\pm$ or $z^\pm$ doesn't change its length, while conjugates by $b^\pm$ 
have length 3.  Thus $\Av(a)=\frac{1+1+1+1+3+3}6=\frac 53$.  The case of $b$ is similar.
That means that 
$\kappa=-2/3$ on the sphere, or $-3/7$ on the ball.  
Thus $F_2\times \Z$ has {\em some points of zero curvature and some points of negative curvature} 
(and no points of positive curvature) in its standard generators, as you'd expect in a 
discrete model of ${\rm Tree}\times \R$.
\end{Example}

%
%
%

\begin{Example}[General RAAGs]\label{ex:RAAG}

Right-angled Artin groups, or RAAGs, are groups in which the only relations are commuting relations between certain
generators.  This is encoded in a simple undirected graph $\Gamma$ for which the vertices represent generators
and the edges appear between commuting pairs.  For instance the free group $F_n$ is a RAAG defined by a graph with no 
edges, and the previous example $F_2\times \Z$ correspond to three vertices arranged in a path.

We will use the fact that RAAGs have a natural normal form:  a string is geodesic iff it does not contain a substring of the form
 $x^{-1}(y_1y_2\dots y_n)x$ where $x, y_i$ are generators and $x, y_i$ commute for all $i$. 

\begin{Proposition}[Curvature dichotomy for RAAGs]\label{prop:RAAGs}
	Let $A_{\Gamma}$ be the RAAG on the graph $\Gamma$ with the standard generating set $S = V(\Gamma)^{\pm1}$. Then 
	$\kappa(g)=0$ iff $g$ is central; else, $\kappa(g)<0$.  
\end{Proposition}
\begin{proof}
In light of Proposition~\ref{properties}\eqref{central}, it suffices to show that if $g$ is not central then $\kappa(g)<0$.
Fix such a $g$.
We will show that for all generators $t\in S$, 
 $|t^{-1}gt|+|tgt^{-1}| \geq 2|g|$, and there exists a generator making this inequality strict. 
	
Consider the ``prefixes" and ``suffixes" of geodesics:
let ${\rm Pre}(g) = \{v\in S \mid g = vh$ is a geodesic spelling of $g\}$ and 
${\rm Suf}(g) = \{v\in S \mid g = hv$ is a geodesic spelling of $g\}$. 
These are the letters that can be commuted to the front or the back of a geodesic spelling for $g$, respectively.
By the normal forms for RAAGs, these sets are well defined, and 
neither ${\rm Pre}(g)$ nor ${\rm Suf}(g) $ can contain both $v$ and $v^{-1}$. Note that all the letters in ${\rm Pre}(g)$
(respectively ${\rm Suf}(g)$) commute with one another, so the corresponding vertices form a clique in $\Gamma$.
These sets will prove useful in calculating word length. 

For $t\in S$, note that  $|gt| = |g| \pm 1$, with 
 $|gt|=|g|-1$ if and only if $t^{-1}\in {\rm Suf}(g) $.  Similarly, $|t^{-1}g| = |g|+1 \iff t\notin {\rm Pre}(g)$. 
Considering all the possibilities for whether $t^\pm$ are in ${\rm Pre}(g)$ or ${\rm Suf}(g) $ 
one sees that $|t^{-1}gt|+|tgt^{-1}| \geq 2|g|$. 
Finally, since $g$ is not central, we can find a letter $x$ such that $\{x\}\cup {\rm Pre}(g)$ is not a clique, 
and that means that both $xg$ and $x^{-1}g$ are geodesic.  Since both $x$ and $x^{-1}$ can't be in 
${\rm Suf}(g)$,  this gives  $|x^{-1}gx|+|xgx^{-1}| > 2|g|$. 
\end{proof}
\end{Example}

\section{Virtually abelian groups}
We have seen that no group can have positive curvature at all nonidentity points, since  
Proposition~\ref{properties}\eqref{avgen}
observes that $\kappa(a)\le 0$ for $a\in S$.  Now we will consider the consequences of identically zero curvature.

\begin{Theorem}[Flat generators implies virtually abelian]\label{thm:0iffvabelian}\label{thm:vabelian}
For a group $(G,S)$, suppose $\kappa(a)=0$
  for all $a\in S$. Then $G$ is virtually abelian.
\end{Theorem}
\begin{proof}
  Since $\kappa(a)=0$ we have 
  that $\Av(a)=1$ for all $a\in S$. In particular, this 
  means that the action of $G$ by conjugation stabilizes 
  $S$ as a set. Consequently, since $S$ is finite, it follows 
  that the stabilizer of a point $a\in S$ is finite index in $G$. 
Since $G$ is acting by conjugation, $\mathrm{Fix}(a) = Z(a)$. 
  Thus, for every $a\in S$, $Z(a)\stackrel{f.i.}{\le} G$. 
  Since $S$ generates $G$, it follows that $Z(G) = \cap_{a\in S} Z(a)$, 
  and this group is finite index in $G$ and abelian.
\end{proof}

As a corollary, we get that groups of constant zero curvature are virtually abelian. 

One might hope that there is a converse, namely that given a virtually abelian group we can find a generating set where every point has $\kappa=0$. We will see soon that this is not possible; however, we do have the following. 

\begin{Theorem}[Virtually abelian implies virtually flat]\label{VA-genset}\label{thm:vabelian2}
  Let $G$ be virtually abelian, so that $H\le G$, $H\cong \Z^n$ is a finite-index normal free abelian subgroup of $G$. 
  Then there exists a generating set $S$ of
  $G$ such that $\kappa(h)=0 \quad \forall h\in H$.
\end{Theorem}
\begin{proof}
	There exists a finite group $F$ such that $G$ fits into the sequence
	$$H\hookrightarrow G \twoheadrightarrow F. $$
	
	Since $G$ acts on $H$ by conjugation, there is a well defined map $\phi\colon F\to Out(H) = GL_n(\Z)$. 
	Let $U$ be a set containing a lift of each non-identity element of $F$ and close it under inversion. Let $T$ be a finite generating set for $H$.
	Let $V = \{uvw\mid u, v, w\in U,  uvw\in H\}\cup \{uv\mid u, v\in U, uv\in H\}\cup T$. 
	and let $T' = \{u^{-1}vu\mid t\in V, u\in \langle U\rangle\}$ be the set of conjugates of elements of $T$ by elements of $\langle U\rangle$. This is a finite set since the orbits of $\langle U\rangle$ in $H$ are finite. 
	Let $S = U\cup T'$; by construction $S$ generates $G$. 
	We will show that conjugating by $s\in S$ does not change the word length of any $h\in H$.
	
	Let $x_1\dots x_l$ be a spelling for $h$ that is geodesic with respect to $S=U\cup T'$. Since $T'$ is preserved by the conjugation action of $U$ we can assume that all but the first $k$ letters are in $T'$. We will show that $k=0$. 
	
	Consider the word $x_1\dots x_k$, note that since this represents an element of $H$ we know that $k\neq 1$. First we can see that no subword of the form $x_ix_{i+1}$ gives an element of $H$ as we could replace this with a generator from $T'$. Looking at the final subword $x_{k-1}x_k$ we can see there is an element $y$ such that $yx_{k-1}x_k$ is in $H$. We can then replace $x_{k-1}x_k$ by $y^{-1}t$ for the element $t = yx_{k-1}x_k\in T'$. This allows us to reduce $k$ by 1 without changing the length of our word. We can thus keep reducing until we arrive at the case that $k=1$ giving a contradiction.
	
	Since $H$ is a normal subgroup, $T'$ is contained in $H$, which is abelian. 
	Therefore generators in $T'$ commute with $h$, so conjugation by $T'$ generators does not change the word length.
	
	Next consider conjugation by $u\in U$. Let $x_i^u=u^{-1}x_iu$, and note that this is in $T'$ so its word length is 1 with respect to $S$ generators. We have a spelling for the  group element $u^{-1}hu$, namely $x_1^u\dots x_l^u$, so $|u^{-1}hu|\le |h|$.
	Conversely, if $y_1\dots y_m$ is a word representing $u^{-1}hu$, then $y_1^{u^{-1}} \dots y_m^{u^{-1}}$ is a word representing $h$, so $|h|\le |u^{-1}hu|$.
	Thus conjugation fixes the length of every word in $H$ and the curvature $\kappa(h)=0$ for all $h\in H$.
\end{proof}	

In the case where the group is virtually $\Z$, we can remove the dependence on generating set.

\begin{Proposition}\label{prop:zhaszeroes}
If $N$ is a normal infinite cyclic subgroup of $(G,S)$, 
then $\kappa(g) = 0$ for all $g\in N$. 
\end{Proposition}
\begin{proof}
	Take an element $g\in N$ and any generator $s\in S$. Since $N\cong \Z$ is normal we see that conjugation by $s$ gives an automorphism of $\Z$ which is either the identity or inversion. In either case this preserves the length of the element $g$. 
\end{proof}

On the other hand, the dependence on generators obstructs the possibility of a complete converse to Theorem \ref{thm:0iffvabelian}, which we see in the next two results.

\begin{Proposition}\label{virtabelpos}\label{prop:vabelian3}
	There is a virtually abelian group $(G,S)$ such that every finite index free abelian subgroup has a point of positive curvature and a point of negative curvature.
\end{Proposition}
\begin{proof}
	Let $G = \langle a, b, t\mid [a, b], t^{-1}at = ab, t^{-1}bt = a^{-1}, t^6\rangle$ with generating set $\{a, b, t\}^{\pm1}.$ This group is isomorphic to $\Z^2\rtimes \Z/6\Z$. 
	The subgroup $\langle a, b\rangle  = \Z^2$ is a normal abelian subgroup and any finite index free abelian subgroup is a subgroup of this. 
	
	We will show that the element $(ab)^n$ has positive curvature and the element $a^n$ has negative curvature for all $n\ge 1$. 
	
	The element $w=(ab)^n$ can be written as a geodesic using the spelling $t^{-1}a^nt = w$. So this word has length $n+2$. Since $w=(ab)^n\in \Z^2$, conjugating by the generators $a, b$ does not change the length. We must now consider the conjugates by $t$ and $t^{-1}$. Firstly, $twt^{-1} = tt^{-1}a^ntt^{-1} = a^n$ and that has length $n$.
	Finally, $t^{-1}wt = t^{-2}a^nt^2 =t^{-1}(ab)^nt= b^n$ which also has length $n$. We conclude that $\Av(w) = n + \frac{4}{3} < n+2$, thus $\kappa(w)>0$. 
	
	The element $v = a^n$ is a geodesic and so we work with this spelling. Conjugating by elements of $\langle a, b\rangle$ will not change the length. Conjugating by $t$ we get $t^{-1}vt$ is a geodesic and has length $n+2$. Conjugating by $t^{-1}$ we get $tvt^{-1} = b^{-n}$, which is a geodesic and has length $n$. 
	This gives $\Av(v) = n + \frac{1}{3}>n$, thus $\kappa(v)<0$. 
\end{proof}

\begin{Proposition}\label{prop:dihedral} 
Let $D_\infty=\Z_2\ast \Z_2$, which is virtually $\Z$.
This group has no finite generating set giving it constant zero curvature. 
\end{Proposition}
\begin{proof}
	First consider $D_\infty = \Z_2\ast \Z_2$ with the standard presentation $\langle a, b\mid a^2, b^2\rangle$. Every element of this group can be represented uniquely
	and geodesically as a positive word in $a$ and $b$ with no occurrence of $a^2$ or $b^2$. 
(The cyclic group generated by $ab$ has index 2, with $D_\infty = \langle ab\rangle \sqcup a\langle ab \rangle$.)
	
	If $D_\infty$ had a finite generating set $S$ with zero curvature everywhere, then $S$ would be preserved by the conjugation action on itself, i.e.,
	 $s^{-1}ts\in S$ for all $s, t\in S$.  (As in the proof of Theorem~\ref{thm:0iffvabelian}, this is because conjugation can't carry a generator to the identity, so it can only lengthen the generators, 
	 but if any of them strictly increases in length, then $\Av$ will go up and curvature will be negative.)
	
	If every element of $S$ is a word which starts and ends in different letters, then $S$ is contained in the subgroup $\langle ab \rangle$, because
	$ba=(ab)^{-1}\in \langle ab \rangle$.  Since $\langle ab \rangle$ is a proper subgroup, we conclude that 
	there is an element $s\in S$ which starts and ends with the same letter, say $a$. There are two cases:
	 either $s=a$ or $s=awa$ for some nonempty $w$. We will deal with the case that $s=a$; the other case works the same way.
	
	There is at least one more element $t\neq s$ of $S$ since $D_\infty$ has rank 2. 
	Up to replacing $t$ by $t^{-1}$, there are four cases to check: freely reduced spellings
	$t = b, t = bvb, t = bva$ or $t = ava$ for strings $v\in S^*$.
	
	Firstly, if $t=b$, then $t^{-1}st = bab\in S$ and $s^{-1}(t^{-1}st)s = ababa\in S$. Repeatedly conjugating by $t$ and then $s$ we find more elements that must belong to $S$, contradicting finiteness. The proof for $t=bvb$ is exactly the same. 
	
	If $t=bva$, then we can look at $t^{-i}st^i = (av^{-1}b)^ia(bva)^i$ which is reduced and must be in $S$ for all $i$, showing again that $S$ would be infinite. 
	
	Finally, if $t = ava$, then we can start by conjugating by $a$ to see that $v\in S$. Since $t$ was a reduced word, we see that $v$ starts and ends with $b$,
	which is a case that has already been treated.
\end{proof}

\section{Nilpotent groups}

Finitely generated groups are called nilpotent when nested commutators are killed after a certain depth:
we call $s$ the step of nilpotency if it is the  index at which the lower central series terminates;
in other words, if $G_1=G$ and $G_{i+1}=[G_i,G]$, we have 
$$1=G_s\trianglelefteq \dots \trianglelefteq G_1=G.$$
This means that abelian groups are 1--step nilpotent.  
The next simplest example, the free abelian group of rank two and step two---the Heisenberg group---has standard presentation $H(\Z)=\langle a,b : [a,b] ~\hbox{\rm is central} \rangle$.
From a geometric group theory point of view, nilpotent groups have aspects of flat geometry and
aspects of hyperbolic geometry---for instance, in \cite{duchin-mooney}, it is shown that a positive share of Heisenberg geodesics have a fellow-traveling property associated with hyperbolic geometry, and a positive share have the same instability seen in abelian groups.  On the other hand, they do not have the 
CAT(0) property that is the dominant notion of non-positive curvature in geometric group theory:  they act geometrically on Lie groups with arbitrarily separated pairs of conjugate points, characteristic of positive curvature.  (See also \cite[p252]{bridson_metric_1999}.)  

In this section we see that this is tracked by conjugation curvature.
First, note that every nontrivial nilpotent group has a nontrivial center (including at least $G_{s-1}$),
which means that it has some points with $\kappa=0$.  
We will establish that nilpotent groups can also have infinitely many points of positive and negative curvature,
by finding all three cases ($\kappa=0,>0,<0$) with positive density in the Heisenberg group.

Blach\`ere computed the 
exact formula for the standard word metric on $H(\Z)$  in \cite{blachere} in terms of matrix-entry coordinates.
We will instead use Mal'cev coordinates: even though $a,b$ suffice to generate the group, adding 
$c=[a,b]$ gives a nice normal form $a^A b^B c^C$, with integer exponents $A,B,C$.  
All finitely generated nilpotent groups admit similar Mal'cev coordinates, with the first level corresponding 
to the abelianization $G_1/G_2$ and the subsequent levels corresponding to subgroups with higher degrees of polynomial distortion.  

In the $H(\Z)$ case, we can use coordinates $(A,B,C)\in \Z^3$ to denote $a^A b^B c^C$;
the $\langle c\rangle$ subgroup is quadratically distorted.
Then the word metric obeys the following formula:
if $ A> B>0$ and  $C>0$, then
\begin{equation}\label{heisenlength}\tag{$\star$}
|a^A b^B c^C|=\begin{cases} 
2\lceil C/A \rceil +A+B, & C\le A^2-AB; \\ 
2\lceil 2\sqrt{C+AB}\rceil -A-B,& C\ge A^2-AB. \end{cases}
\end{equation}
Let us call the first case ($C\le A^2-AB$) the {\em low-height} case;
then the {\em high-height} case is $C\ge A^2-AB$.
A distinction like this will work in general nilpotent groups, with the low-height case corresponding to words admitting efficient spellings with first-level Mal'cev letters.

\begin{Theorem}\label{thm:heis}
The Heisenberg group $H(\Z)$ with standard generators $S=\{a,b\}^\pm$ has 
 a positive proportion of points with positive, negative and 
zero curvature, respectively:  there exists $\epsilon>0$ such that
$$\epsilon<\frac{ \#\{g\in B_n : \kappa(g)=0\}}{\#B_n} < 1-\epsilon,$$ 
and the same is true for $\kappa>0$ and $\kappa<0$.
\end{Theorem}

\begin{proof}First let us describe the ball $B_n$ in terms of the coordinates $(A,B,C)$ of its elements.
For each  $(A,B)$ satisfying $|A|+|B|\le n$, the point $(A,B,C)$ belongs to $B_n$ iff $C$ satisfies a certain
quadratic polynomial inequality in $A,B$---this can be read off of the word-length formula with a little bit of 
work (because of the square root appearing in the high-height case and the quadratic bound for the low-height case), 
or is described directly in \cite{duchin-mooney}.

Now let $w=a^A b^B c^C\in B_n$ and note what happens when we conjugate by generators $a,b$.  We have
$$a^{-1}(a^A b^B c^C)a= a^A b^B c^{C-B}; \quad
a(a^A b^B c^C)a^{-1}=a^A b^B c^{C+B}; $$
$$b^{-1}(a^A b^B c^C)b= a^A b^B c^{C+A}; \quad
b(a^A b^B c^C)b^{-1}=a^A b^B c^{C-A}.$$

At low height, for $A,B$ fixed, we will show that 
the sign of $\kappa$ only depends on $C \mod A$.  
To see this, note that conjugating by $b^\pm$ changes $C$ by adding $\pm A$.  
By \eqref{heisenlength}, each of these two 
conjugates changes $|w|$ by exactly 2 in opposite directions, which cancels out in the averaging.  Thus only conjugates
by $a^\pm$ contribute to the sign of $\kappa$.  

In particular if $C$ is a multiple of $A$ then  $C/A$ is a whole number, so $\lceil C/A \rceil$ is preserved by $C\mapsto C-B$
and increased by $C\mapsto C+B$, which means that the curvature is negative overall.  
On the other hand, if $C=kA+1$, then $\lceil C/A\rceil$ is preserved by $C\mapsto C+B$ and 
decreased by $C\mapsto C-B$, for positive curvature overall.  
Generally the sign of $\kappa$ in this case (low height; $A,B$ fixed; $C$ varying) only depends on $C$ mod $A$.
To be precise, suppose $C=kA+r$.  Then 
$$\lceil (C+B)/A \rceil = \begin{cases} k+2, & r>A-B\\ k+1, & r\le A-B \end{cases} \quad {\rm and} \quad 
\lceil (C-B)/A \rceil = \begin{cases} k+1, & r>B\\ k, & r\le B. \end{cases}
$$

Within the conditions $A>B>0$, $C>0$, let us further restrict to the sector $\frac 15 A \le B < \frac 25 A$; clearly a positive
proportion of the points in the ball satisfy these inequalities.  Then of the possible remainders $r$ for $C$, 
at least $1/5$ satisfy $1\le r \le B$ (call this Case 1), at least $1/5$ satisfy $A-B<r\le A-1$ (Case 2), and at least $1/5$ satisfy $B<r\le A-B$
(Case 3).  
In all three cases, $\lceil C/A\rceil=k+1$.  In Case 1, $\lceil (C\pm B)/A\rceil= k+1,k$, so $\kappa>0$.
In Case 2, $\lceil (C\pm B)/A\rceil= k+2,k+1$, so $\kappa<0$.  And in Case 3, $\lceil (C\pm B)/A\rceil= k+1,k+1$, so $\kappa=0$.

We have established that the sign of the curvature repeats periodically (mod $A$) at low height.
To complete the proof, we must be sure that for most of the $(A,B)$ in the ball of radius $n$, the low-height values obtained by $C$ 
complete at least one period with respect to $A$, and that the low-height case has positive measure.
First we check that $A^2-AB\ge A$ (the low-height inequality
is satisfied for a full period): this is true for $A\ge 5$ because $\frac 15 A \le B < \frac 25 A$.  
Finally, note that $\int_0^n \int _0^{n-A} (A^2-AB) \ dB\ dA = \frac{n^4}{24}$, which has positive measure since $|B_n|$
is bounded above and below by multiples of $n^4$.  
\end{proof}

\begin{Remark}
At high height, for $A,B$ fixed, there is a  pattern of positive and negative curvature points appearing in an interval of fixed length
surrounding each value of $C$
for which $C+AB=k^2$ is a perfect square.  Outside of these intervals, $\kappa=0$.

This means that in a fixed fiber $(A,B,*)$ of the projection map $\pi\bigl( (A,B,C)\bigr)=(A,B)$, almost every point has zero curvature; on the other hand, we've seen that
 over the ball of radius $n$, a positive proportion of points have $\kappa \neq 0$.
\end{Remark}

We note that this proof will carry over to the Heisenberg group with {\em any} generating set.  
We sketch the general proof here.  
There will still be a low-height/high-height distinction in arbitrary generators, with a positive proportion 
of the ball in low height.  (See \cite{duchin-mooney}, where the low-height case is called {\em unstable} and the proportion is given by a certain volume ratio that is shown to be 
positive and rational.)
Each generating set has a defining polygon $Q$ in the plane given by the convex hull of the projection of its generating set.
(For the standard generators, $Q$ is the diamond $|A|+|B|\le n$.)
It is still the case that conjugation by generators induces a linear change in the $C$ coordinate.
We can restrict to a sector of the defining polygon bounded away from its vertices, as we did here, so that there exists $h>0$
for which $(A,B) \in nQ$ and $C\in [0,hn^2]$ implies $(A,B,C)\in B_n$.   This guarantees that the fibers in this sector
over each $(A,B)$ will see a full period with respect to $A$ and $B$.  
The word length will still be given in  $(A,B,C)$ by a linear function with rounding  in the low-height case.
Thus we will get linear inequalities determining whether $\kappa>0,\kappa<0,$ or $\kappa=0$, with each 
one satisfied for a positive proportion of points in $B_n$.
We expect that an argument like this is adaptable to all finitely generated nilpotent groups via analysis of their Mal'cev bases.

\begin{Remark}
In the 1976 paper  \cite{milnor}, John Milnor shows that 
for any left-invariant Riemannian metric on a nonabelian nilpotent Lie group, there are directions
of strictly positive Ricci curvature and directions of strictly negative Ricci curvature.
In particular, the central direction is positively curved, and non-central directions orthogonal to the commutator subalgebra
are negatively curved.
Note that the asymptotic cone of the word metric $(H(\Z),\{a,b\}^\pm)$ is the Lie group $H(\R)$ with the (left-invariant)
$L^1$ Carnot-Carath\'eodory metric, which is not Riemannian but only sub-Finsler.  
The work here shows that this metric on $H(\R)$ behaves quite differently from the Riemannian ones studied by Milnor:
the $x$-direction is still negatively curved, as in the Riemannian case 
(because the conjugate of $a^n$ by $b^\pm$ has length $n+2$ for $n\ge 1$),
but for us the $z$-direction (the center) has zero Ricci curvature.  
\end{Remark}

\section{Positively curved directions in CAT(0) and hyperbolic groups}

As mentioned above, the CAT(0) condition is a metric-space version of non-positive curvature and 
$\delta$--hyperbolicity is an asymptotic definition of negative curvature; both are widely used in geometric group theory.  While conjugation curvature should be expected to relate nicely to these other notions, it would be too much to ask for the complete absence of points of positive conjugation curvature.  This is unsurprising:  hyperbolicity in particular allows for infinitely many instances of small-scale positive curvature in a metric space.  Nonetheless we will briefly illustrate the point with some constructions intended to illuminate how to work with 
$\kappa$.

This section uses standard constructions like HNN extensions and $C'(\alpha)$ small cancellation, for which a good reference is Bridson and Haefliger \cite{bridson_metric_1999}.

\subsection{A CAT(0) example}

We will build a family of CAT(0) groups which have points of positive curvature.
First we recall the following Proposition from \cite{bridson_metric_1999}.

\begin{Proposition} [Bridson-Haefliger p353, Prop 11.13]
Let $X$ and $A$ be locally CAT(0) spaces. If $A$ is compact and $\phi, \psi\colon A\to X$ are local isometries, then the quotient of $X\sqcup A$ by the equivalence relation generated by $[(a, 0)\sim \phi(a)$ and $(a, 1)\sim \psi(a)]$ is locally CAT(0).
\end{Proposition}

We can now prove the following. 

\begin{Theorem}[Positive curvature in CAT(0) groups]\label{thm:cat0}
For fixed $n\ge 3, k\ge 2$, the  group
$$G_{n, k} = \langle x, y, a_1, \dots, a_k\mid a_i x^na_i^{-1} = y^{n-2}, a_i^{-1}x^na_i = y^{n-2} \ 
\forall i\rangle$$
	is CAT(0). In each such group,  $\kappa(x^{jn})>0$ for all $j=1,2,\dots$.
\end{Theorem}
\begin{proof}
	Take a rose $R$ with 2 petals one of length $n$ and one of length $n-2$. Given a rose $R'$ with two petals both of length $n(n-2)$ there are two local isometries $\phi, \psi \colon R'\to R$ wrapping the loops around the loops in $R$ the appropriate number of times. Doing this $k$ times will give a non-positively curved quotient by the above proposition. A simple application of Seifert--van Kampen shows that the fundamental group of this space is $G_{n, k}$.
	
	We claim that $a_1^{-1}y^{j(n-2)}a_1$ is a geodesic spelling for $x^{jn}$. 
	To see this consider the map 
$\phi\colon G_{n, k}\to \Z$ given by $a_i\mapsto 0$, $x\mapsto n-2$, and $y\mapsto n$.
	Let $w$ be a word in $G_{n, k}$. 
	If $|w| < k$, then $\phi(w)<kn$. 
	Suppose now that $v$ is a geodesic spelling for $x^{jn}$. 
	Using HNN normal forms, $v$ has normal form $x^{jn}$, so there must be a ``pinch,'' i.e., $v$ decomposes as $w_1 a_i^{\epsilon}w_2a_i^{-\epsilon}w_3$, where $w_2\in \langle x^n, y^{n-2}\rangle.$ 
	Since $\phi(a_i) = 0$ we can see that $\phi(v) = \phi(w_1w_2w_3) = jn(n-2)$. 
	Thus, $|w_1w_2w_3| \geq j(n-2)$.
	We conclude that $|v| \geq j(n-2)+2$. 

	Thus we have  $|x^{jn}|=j(n-2)+2$, and  by construction, for all $i$, 
	$|a_i^{-1}x^{jn}a_i|=|y^{j(n-2)}|=j(n-2)$, and similarly for $a_i^{-1}$. 
	Additionally, $|xx^nx^{-1}|= |x^{-1}x^{jn}x|=|x^{jn}| = j(n-2) +2$, and 
	by the triangle inequality, $|y^{-1}x^{jn}y|\le j(n-2)+4$, 
	and similarly $|yx^{jn}y^{-1}|\le j(n-2)+4$. All together, we get
	\begin{align*}
	\Av(x^{jn}) &\le \frac{1}{2k+4} \bigl(2k(j(n-2))+2(j(n-2)+2)+2(j(n-2)+4)\bigr) \\
	&= \frac{1}{2k+4}\bigl((2k+4)(j(n-2)) + 12\bigr)\le j(n-2)+2.
	\end{align*}
	This gives 
	$\kappa(x^{jn}) = 1-\frac{1}{j(n-2)+2}(\Av(x^{jn})\bigr)>0$, 
	as desired.
\end{proof}

\subsection{A hyperbolic example}

Before we turn to a hyperbolic group construction, we start with a useful computation for products.
Let us adopt the notation that if $S_1,S_2$ are generating sets for $G_1,G_2$, respectively, then 
$S_1 \boxtimes S_2$ will denote the {\em split} generating set $S = (S_1\times\{e\})\cup (\{e\}\times S_2)$ for the product
$G_1\times G_2$.

\begin{Proposition}[Product formula]\label{prop:productformula}
	Let $G_1$ and $G_2$ be groups with generating sets 
	$S_1, S_2$ respectively. Let $S=S_1\boxtimes S_2$
	 be the split generating 
	set for $G_1\times G_2$. Then $$\Av\bigl( (x, y)\bigr) = 
	\frac{|S_1|\sdot\Av(x)+|S_1| \sdot |y|+|S_2| \sdot |x|+|S_2|\sdot\Av(y)}{|S_1|+|S_2|}.$$
\end{Proposition}
\begin{proof}
	We will compute $(|S_1|+|S_2|)\Av\bigl( (x, y)\bigr)$. This is just the sum of the lengths of all conjugates of $(x, y)$. 
	\begin{align*}
		\sum_{s\in S}|s^{-1}(x, y)s| &= \sum_{s_1\in S_1}|(s_1^{-1}, e)(x, y)(s_1, e)| + \sum_{s_2\in S_2}|(e, s_2^{-1})(x, y)(e, s_2)|\\
		&=\sum_{s_1\in S_1}|(s_1^{-1}xs_1, y)| + \sum_{s_2\in S_2}|(x, s_2^{-1}ys_2)|\\
		&=\sum_{s_1\in S_1}(|s_1^{-1}xs_1|+|y|)+\sum_{s_2\in S_2} (|s_2^{-1}ys_2|+|x|)\\
		&=|S_1| \sdot |y|+|S_2| \sdot |x|+\sum_{s_1\in S_1}|s_1^{-1}xs_1|+\sum_{s_2\in S_2}|s_2^{-1}ys_2|\\
		&=|S_1|\sdot\Av(x)+|S_1| \sdot |y|+|S_2| \sdot |x|+|S_2|\sdot\Av(y).\qedhere
	\end{align*}
\end{proof}

\begin{Theorem}[Positive curvature in hyperbolic groups]\label{thm:hyperbolic}
	There is a non-elementary hyperbolic group and a sequence of elements $w_i\to\infty$ such that $\kappa(w_i)>0$.
\end{Theorem}

\begin{proof}
	Consider two random reduced spellings $x, y\in F(a,b)=\langle a,b \mid \emptyset \rangle$ of length 
	$l$ and let $w_1 = axa$ and $w_2 = y$.  
	Then for any $\epsilon>0$, the following properties hold with positive probability as $l\to\infty$:
	\begin{itemize}
		\item $w_1$ and $w_2$ are cyclically and freely reduced,  and 
		\item $\braket{a,b\mid w_1, w_2}$ is a $C'(\epsilon)$ small cancellation group.  
	\end{itemize}
	Fix some such $w_1,w_2$.

	Consider the group $G=\langle a, b, t\mid t^{-1}w_1 t = w_2\rangle$, and let $S=\{a,b,t\}^\pm$
	denote its generators.
	By the above hypothesis, 
	we can fix $\epsilon=1/6$ and choose $w_1,w_2$ as relators 
	presenting a $C'(1/6)$ small cancellation group, which is necessarily
	a non-elementary hyperbolic group. By small cancellation, 
	any trivial word contains at least half a relation, 
	therefore $w_1$ is a geodesic spelling in the group $G$.  We have $|w_1|=l+2$ and 
	$|w_2|=l$.

Consider the lengths of the 6 conjugates of $w_1$, recalling that $w_1$ starts and ends with $a$.
	\begin{align*}
	|a^{-1}w_1a|&=l+2, 	&|t^{-1}w_1t|&=l, 	&|b^{-1}w_1b|&\le l+4,\\
	|aw_1a^{-1}|&=l+2, 	&|tw_1t^{-1}|&\le |w_1|+2=l+4, 	&|bw_1b^{-1}|&\le l+4.
	\end{align*}	
Thus we have $\Av(w_1) \le \frac{6l+16}{6} = |w_1|+\frac{2}{3}$.

The hyperbolic group we are constructing is $\left( G\times \Symm(4), \ S \boxtimes \Spos\right)$, 
where $\Spos$ is the generating set from Proposition \ref{symm-sign} that was designed to produce positive conjugation curvature.  This is hyperbolic because $G$ has finite index.  
Note that $|\sigma|=2$ in these generators.
From the proof of Proposition \ref{symm-sign}, we know that $\Av(\sigma) = \frac{26}{21}$.
	We can now compute the curvature 
	$\kappa \bigl( (w_1,\sigma)\bigr)$ using the formula from 
	Proposition \ref{prop:productformula}, obtaining 
$$\Av\bigl( (w_1,\sigma)\bigr) \le \frac{6l+16+12+21(l+2)+26}{27} = \frac{27l+96}{27} = l+\frac{32}{9}<l+4.$$
This gives us positive  curvature: $\kappa\bigl( (w_1,\sigma)\bigr) = 1- \frac{1}{l+4}\Av\bigl( (w_1,\sigma)\bigr)>0$. 
	
	To get an infinite sequence of words, let 
	$v_n=w_1t^{-1}w_2^n t w_1$. This is a geodesic word of length 
	$nl+2+2(l+2)$ in $(G,S)$.  
	Conjugating $v_n$ by $a$ or $a^{-1}$ preserves length,  conjugating by 
	$t$ will shorten it by 2 and the other three conjugations (by $t^{-1},b,b^{-1}$) will 
	lengthen it by at most 2. Thus, we get $\Av(v_n)\le |v_n|+\frac 23$.  
	On the other hand, so we get 
	$$\Av\bigl( (v_n,\sigma)\bigr)\le |v_n|+ \frac{38}{27}<|v_n|+2.$$ 
	Thus we can see that we have $\kappa\bigl( (v_n, \sigma)\bigr)>0$. 
\end{proof}

This is not a glitch, but rather conjugation curvature is picking up something like a Sphere$\times\R$ built into the group, which is allowed by hyperbolicity.

\section{Behavior under isometric embeddings}

In this section we show that the designer generating sets $\Spos,\Sneg$ can induce bad embeddings.
Recall the notation $S_1\boxtimes S_2 = (S_1\times\{e\})\cup (\{e\}\times S_2)$.

\begin{Theorem}[Positively curved embedding]\label{thm:pos-embed}
For any $(G,S)$ and sufficiently large $n$, there is an isometric embedding $i\colon G\to G\times \Symm(n)$ such that $\kappa \bigl(i(g)\bigr)>0$ for all $g\in G$ with respect to a generating set $S\boxtimes \Spos$. 
	Moreover, $i$ is a quasi-homomorphism:  $d\bigl(i(gh), i(g)i(h)\bigr)\leq 2$. 
\end{Theorem}
\begin{proof}
Let  $\sigma = (1\dots n) \in \Symm(n)$ and 
let  $\Spos=\Symm(n)\ssm\{e, \sigma\}^\pm$ be the generating set for $\Symm(n)$ designed 
to make $\kappa(\sigma)>0$, as in  Proposition \ref{symm-sign}. 

	Consider the  map $i\colon G\to G\times \Symm(n)$  defined by $i(g) = (g, \sigma)$.
Then by Proposition \ref{prop:productformula} we can write 
	\begin{align*}
		\Av\bigl( (g,\sigma)\bigr)=&\frac{(|S|+|\Spos|)|(g,\sigma)|+|S|(\Av(g)-|g|)+|\Spos|(\Av(\sigma)-|\sigma|)}{|S|+|\Spos|}\\
		 = &|(g,\sigma)|+\frac{|S|(\Av(g)-|g|)+|\Spos|(\Av(\sigma)-|\sigma|)}{|S|+|\Spos|}.
	\end{align*}
To have positive curvature at the point $(g, \sigma)$ we require $|S|(\Av(g)-|g|)+|\Spos|(\Av(\sigma)-|\sigma|)$ to be negative.
Since $S$ is fixed and $\Av(g)-|g|\leq 2$, we must pick $n$ large enough so that $|\Spos|(\Av(\sigma)-|\sigma|)<2|S|$. 
From Proposition \ref{symm-sign}, $|\Spos|(\Av(\sigma)-|\sigma|) = n!+2n-6 - 2(n!-3) = -n!+2n$. Taking $n$ large we can ensure 
$-n!+2n+2|S|<0$. Thus for sufficiently large $n$ we have $\Av\bigl( (g, \sigma)\bigr)<|(g, \sigma)|$ and $\kappa\bigl( (g,\sigma)\bigr)>0. $
\end{proof}

In an exactly similar fashion, we get a corresponding statement for negative curvature.

\begin{Theorem}[Negatively curved embedding]\label{thm:neg-embed}
For any $(G,S)$ and sufficiently large $n$, there is an isometric embedding $i\colon G\to G\times \Symm(n)$ such that $\kappa \bigl(i(g)\bigr)<0$ for all $g\in G$ with respect to a generating set $S\boxtimes \Sneg$. 
	 Moreover, $i$ is a quasi-homomorphism: $d\bigl(i(gh), i(g)i(h)\bigr)\leq 1$. 
\end{Theorem}

\begin{proof}
Let $\Sneg(n) = \{(12), (23), \dots(n\! -\! 1 \  n), \sigma^\pm \}$
 be the generating set for $\Symm(n)$ designed to make $\kappa(\sigma)<0$, as in Proposition \ref{symm-sign}.
 Take the embedding $i\colon G\to G\times \Symm(n)$  defined by $i(g) = (g, \sigma)$, as above.
	To have negative curvature at the point $(g, \sigma)$ we require 
	$\frac{|S|(\Av(g)-|g|)+|\Sneg|(\Av(\sigma)-|\sigma|)}{|S|+|\Sneg|}$ to be greater than 0. 
	Since $S$ is fixed and $\Av(g)-|g|\geq -2$, we must pick $n$ large enough so that $|\Sneg|(\Av(\sigma)-|\sigma|)>-2|S|$. 
	From Proposition \ref{symm-sign}, $|\Sneg|(\Av(\sigma)-|\sigma|) = 2n - 2$. Taking $n>|S|$ we get that $2n - 2-2|S|>0$. 
	Thus for sufficiently large $n$ $\Av\bigl( (g, \sigma)\bigr)>|(g, \sigma)|$ and $\kappa\bigl( (g, \sigma)\bigr)<0.$
\end{proof}

In both examples above, the embedding is quasi-dense. If we relax this requirement, we can make the negative curvature embedding a homomorphism. 

\begin{Theorem}[Negatively curved homomorphic embedding]\label{thm:neg-embed2}
For any group $(G,S)$, there exists a group $(H,T)$ and a homomorphism $\phi\colon G\to H$ such that
	\begin{enumerate}
		\item $\phi$ is an isometry, and 
		\item $\kappa\bigl(\phi(g)\bigr)<0$ for all $g\in G\smallsetminus\{e\}.$  
	\end{enumerate}
	In fact, $\kappa(h)<0$ for all $h\in H\smallsetminus\{e\}$. 
\end{Theorem}
\begin{proof}
	Let $F_n$ be the free group of rank $n$ with the standard generating set $U$. We defer the choice of $n$ until later. Let $H = G\ast F_n$ with the generating set $T = S\cup U$. 
	Then $\phi(g) = g$ defines a homomorphism which is an isometry onto its image. 
	
	We must now compute $\Av(g)$ for $g\in G$. Let $u$ be a generator of $F_n$. The word $u^{-1}gu$ is a geodesic in the group $G\ast F_n$ and so this element has length $|g|+2$. Thus, $$\Av(H, T,g) = \frac{|S|\sdot\Av(G,S,g)+2n(|g|+2)}{|S|+2n}.$$
	Since $\Av(G,S,g)\geq |g|-2$, taking $n>\frac{|S|}{2}$, we have $\Av(H, T,g)>|g|.$ Thus, $\kappa(g)<0$. 
	
	Let $h$ be an element of $H$. If one geodesic representative starts with an element of $G$, then all geodesic representatives will start with a letter of $G$. Similarly, if any geodesic representative ends with a letter of $G$, then all geodesic representatives end with a letter of $G$. 
	
	We now split into cases. Firstly, consider the case that both the first and last letters of $h$ are in $G$. In this case we see that every element of $F_n$ lengthens this word. Thus taking $n$ large enough we can ensure such elements have negative curvature. 
	
	Secondly, for an element that has a geodesic representative where both the first and last letter are in $F_n$ we see that all elements of $S$ will lengthen this element. From the calculation in Example \ref{freecurv} we see that for any $n\geq1$ these elements will have negative curvature. 
	
	Finally consider the case that a geodesic representative starts with a letter of $S$ and ends with a letter of $U$. Conjugating by elements of $S$ will either preserve the length or make the word longer. Conjugating by $U$ will have a similar effect. Taking $n\geq 2$ we see that there is an element of $U$ which lengthens the word and thus we get negative curvature. 
	
This completes the proof that the group has negative curvature everywhere. 
\end{proof}

Recall the well-known fact that no  hyperbolic group has a subgroup isomorphic to $\Z^2$.
So using the above construction with $G=\Z^2$, we have constructed a non-hyperbolic group
with every non-identity element negatively curved.  
So negative conjugation curvature does not signal $\delta$--hyperbolicity.  
As we will discuss below, it does signal exponential growth.

\section{Curvature outside a ball}
Since generators are always nonpostively curved (see Prop~\ref{properties}\eqref{avgen}), no nontrivial group can have 
positive curvature  at all points.
What happens when we assume that $\kappa(g)>0$ for all 
$g$ outside a ball of finite radius?
\begin{Theorem}[Positive curvature outside a ball] \label{Finite_Outside_Ball}\label{thm:finiteness}
  Suppose for some $(G,S)$ there exists 
  $R\in \N$ such that $\kappa(g)>0$ for all $g\not\in B_R$. Then $G$ is a finite group.
\end{Theorem}
\begin{proof}
Define $A_n = S_{2n}\cup S_{2n+1}$ to be an annulus in the Cayley graph. 
For all $g\in G$, $s\in S$, let $w_{g, s} = |g| - |s^{-1}gs|$ measure the amount that $g$ is shortened
by conjugation by $s$.
	 Let $c_n = \sum\limits_{S_n\times  S}w_{g, s}$. Noting that for $g\in S_n$ we have $|g| = n$, the following chain of equalities relating $c_n$ to $\kappa(g)$. 
	
	\begin{align*}
		c_n &= \sum_{(g, s)\in S_n\times S}w_{g, s}
			= \sum_{(g, s)\in S_n\times S}\Big(|g| - |s^{-1}gs|\Big)
			= \sum_{g\in S_n} \Big( |S|\sdot |g| - \sum_{s\in S}|s^{-1}gs| \Big) \\
			&= \sum_{g\in S_n} \Big(|S|\sdot |g| - |S|\sdot \S_1(g)\Big)
			= n|S|\sum_{g\in S_n} \Big(1 - \frac{\S_1(g)}{n}\Big)
			= n|S|\sum_{g\in S_n} \kappa(g).
	\end{align*}
	Thus for $n>R$ we have $c_n>0$. 
	
Define the following sets: 
\begin{align*}
K_n = \{(g, s)\in A_n\times S \mid  s^{-1}gs\in A_{n-1}\} \ ; \qquad 
J_n = \{(g, s)\in A_n\times S \mid  s^{-1}gs\in A_{n+1}\};\\
L_n = \{(g, s)\in S_n\times S \mid  w_{g,s}=1\} \ ; \qquad 
U_n = \{(g, s)\in S_n\times S \mid  w_{g,s}=-1\}.
\end{align*}

	Let $k_n = \sum_{K_n}w_{g, s}$. 
	The generating set $S$ is closed under inversion so,
$\sum_{L_{n+1}}w_{g, s} = -\sum_{U_{n}}w_{g, s}$ and 
$\sum_{J_n}w_{g, s} = -\sum_{K_{n+1}}w_{g, s}=-k_{n+1}$.
	
	For $n\gg R$ we have
	\begin{align*}
		0<c_{2n}+c_{2n+1} &= \sum_{(g, s)\in S_{2n}\times  S}w_{g, s} + \sum_{(g, s)\in S_{2n+1}\times  S}w_{g, s}\\
		&= \sum_{(g, s)\in K_n} w_{g, s} + \sum_{(g, s)\in L_{2n+1}}w_{g, s} + \sum_{(g, s)\in U_{2n}}w_{g, s} + \sum_{(g, s)\in J_n} w_{g, s}\\
		&= \sum_{(g, s)\in K_n} w_{g, s} + \sum_{(g, s)\in J_n} w_{g, s}
		= \sum_{(g, s)\in K_n} w_{g, s} - \sum_{(g, s)\in K_{n+1}} w_{g, s}
		= k_n- k_{n+1}.
	\end{align*}
We conclude that for indices bigger than $R$, the sequence $(k_n)$ is strictly decreasing.

	Positive curvature ensures that every element of $S_{2n}$ has at least one generator which conjugates it into $A_{n-1}$, so we have  $$|S_{2n}|\leq k_n \leq 2|K_n|\leq 2|S|\sdot |A_n|.$$
	
	We are now ready to assemble these facts. If $G$ has linear growth, then it is virtually $\Z$ and we can appeal to Proposition \ref{prop:zhaszeroes}. 
	Thus assume $G$ is a group with superlinear growth. Since $|S_{2n}|\to\infty$, we see that for 
$m\gg n \gg R$ we get $$|S_{2m}|\leq k_m< k_n\leq 2|S|\sdot |A_n|\leq 4|S| \sdot |S_{2n+1}|,$$ 
giving a contradiction. 
\end{proof}

\section{Directions for future study} 

We have seen here that positive curvature outside of a ball signals finiteness for the group, i.e., that the growth function becomes constant.  
At the time of circulating this paper, we conjectured that negative curvature outside of a ball signals exponential growth, which is now proved in a new preprint by Nguyen--Wang \cite{expgrowth}.  We suspect that stronger group-theoretic conditions are also implied by negative conjugation curvature, such as non-amenability.

In this paper, we have established that interesting results are already present at $r=1$, i.e., 
conjugation by generators.
A future direction to explore is adjusting the radius parameter, making spherical comparisons for fixed $r>1$.
In particular, for hyperbolic groups, studying $\kappa_r$  for sufficiently large $r$ will give negative curvature. 
Ollivier shows in \cite{ollivier2009} that one can choose $r=r(\delta)$ for a hyperbolic group so that his transportation curvature is uniformly negative.  Since conjugation curvature lower-bounds transportation curvature, that finding carries over to the present definition for $r=r(S)$.

As a positive sign for this direction of inquiry, Kropholler--Mallery have recently shown that passing to large $r$ will not wipe out all of the findings of mixed curvature:  they show in 
 \cite{krophollermallery} that the positive-density results for $\kappa=0,<0,>0$ persist for large $r$.
  
The Kropholler--Mallery paper also proves interesting connections between values of $r$ for which 
$\kappa_r>0$ and an invariant called dead-end depth in Cayley graphs. 
Suggestively, hyperbolic groups have uniformly bounded dead-end depth, and we've seen that conjugation curvature is positive for only finitely many radii. 
The property of uniformly bounded dead-end depth also holds true in Cayley graphs of virtually abelian groups and groups with more than one end \cite{lehnert}. 
This leads us to conjecture that conjugation curvature at larger radius will eliminate some of the unwanted sources of  positive curvature.
\begin{Question}
Let $(G,S)$ be a finitely generated virtually abelian group or a finitely generated group with more than one end. 
Is there an $r = r(S)$ such that $\kappa_r(g)\leq 0$ for all $g\in G$?
\end{Question}

Finally, two more very natural tweaks to create a more robust invariant would be evaluating averages on balls rather than spheres, and letting $r$ vary with $|x|$.  Taking $\kappa^\B_r(x)$, where $r$ is ``tuned" as a function of $|x|$, effectively averages over larger selections of group elements. 
For instance it would be interesting to study the $r=|x|$ case or the $r=\sqrt{|x|}$ case, which in particular would lead to negative curvature outside a ball in hyperbolic groups.  These variants could be profitably studied 
both for our conjugation curvature and for Ollivier's transportation curvature, giving a promising way to evaluate
the curvature of large but finite graphs.  Making a principled choice of scale is an excellent question for further study.

\end{document}